\documentclass[12pt]{amsart}
\usepackage{calligra}
\usepackage{amsthm}
\usepackage{amsmath}
\usepackage{amssymb}
\usepackage{showkeys}
\usepackage{amscd}
\usepackage{bbm}
\usepackage[all]{xy}
\usepackage{mathrsfs}
\usepackage[top=72pt,bottom=96pt,left=72pt,right=72pt]{geometry}
\newtheorem{Lemma}{Lemma}[section]

\newtheorem{Definition}[Lemma]{Definition}
\newtheorem{Proposition}[Lemma]{Proposition}

\def\qS{\mathscr{S}}
\def\YM{\mathrm{YM}}
\def\SM{\mathrm{SM}}

\def\ad{\mathrm{ad}}

\def\C{\mathbb{C}}
\def\S{\mathbb{S}}

\def\f{\mathfrak{f}}
\def\qGG{\mathfrak{qGG}}

\def\Hor{\mathrm{Hor}}

\def\dvol{\mathrm{dvol}}

\def\id{\mathrm{id}}

\def\Ker{\mathrm{Ker}}

\def\inv{\mathrm{inv}}
\def\H{\mathrm{H}}

\def\Mor{\textsc{Mor}}
\def\N{\mathbb{N}}

\def\Z{\mathbb{Z}}

\def\triv{\mathrm{triv}}

\def\l{\mathrm{L}}
\def\r{\mathrm{R}}
\def\Id{\mathrm{Id}}

\def\c{\mathrm{c}}

\def\r{\mathrm{R}}

\def\SU{\mathcal{SU}}

\def\T{\mathcal{T}}

\begin{document}
\date{\today}
\title{Yang--Mills--Scalar--Matter Fields in the Quantum Hopf Fibration}
\author{Gustavo Amilcar Salda\~na Moncada}
\address{Gustavo Amilcar Salda\~na Moncada\\
Instituto de Matem\'aticas UNAM}
\email{gamilcar@ciencias.unam.mx}
\begin{abstract}
In this paper we present solutions to the {\it non--commutative geometrical} version of the Yang--Mills--Scalar--Matter theory in the Hopf fibration $S^1\hookrightarrow S^3\rightarrow S^2$ using the $3D$--calculus.
 \begin{center}
  \parbox{300pt}{\textit{MSC 2010:}\ 46L87, 58B99.}
  \\[5pt]
  \parbox{300pt}{\textit{Keywords:}\ Fields, quantum Hopf fibration, canonical quantum principal connection}
 \end{center}
\end{abstract}
\maketitle

\section{Introduction}

In Differential Geometry, the Hopf fibration is perhaps maybe one of the most basic and well--established examples of principal bundles. This bundle is particularly important in Physics since its connections can describe magnetic monopoles. Even more, gauge theory allows us to study matter fields in the presence of magnetic monopoles. In Non--Commutative Geometry there is an analog concept known as the quantum Hopf fibration or the $q$--deformed Dirac monopole bundle \cite{bm}.

This paper aims to show in a concrete  example that the theory presented in \cite{sald1} and \cite{sald2} is non--trivial, i.e., we are going to present a {\it non--commutative geometrical version of magnetic monopoles and its interaction with space--time scalar matter fields}. Unlike examples shown in \cite{sald2} and \cite{sald3} where we used trivial quantum principal bundles; here we will use the quantum Hopf fibration which relates the structure group and the base space in a non--trivial way. The importance of this work lies in the fact it provides a better support of the general theory showing explicit and interesting solutions of the {\it field equations}.

To accomplish our purpose, we will consider the quantum Hopf fibration together with its $3D$--calculus, which represent the {\it non--commutative} counterpart of the {\it classical differential calculus}, but of course, it is possible to use others, for example, the $4D$--calculus like in \cite{lz}. It is worth remarking  (again) that the theory presented here is just an application of the general theory (\cite{sald1}, \cite{sald2}) not an {\it ad hoc} theory created for this space; and this work follows the research line of M. Durdevich shown in \cite{micho1}, \cite{micho2}and \cite{micho3}, among other papers. If the reader wants to check another concrete example of the general theory, we recommend \cite{sald3}.

The paper is organized into five sections. In the second section we are going to develop our study of the quantum Hopf fibration using the $3D$ $\ast$--first order differential calculus ($\ast$--FODC)  to creat the differential calculus in the quantum Hopf fibration shown in \cite{micho3}. This differential calculus arises from the {\it classical} differential geometry and like in that case, one can define a particular quantum principal connection (qpc) which will play the role of the {\it canonical} principal connection on the Hopf fibration. In the third section we will talk about the associated quantum vector bundles (associated qvb) as well as all the necessary conditions to work with Yang--Mills fields and space--time scalar matter fields \cite{sald1}, \cite{sald2}. In the fourth section we are going to present solutions to the fields equations in this space as well as the spectrum of the left and right Laplace-de Rham operator for qvbs \cite{sald2}. The last section is about some concluding comments. It is worth mentioning that this paper is not self--contained, so we strongly recommended to read \cite{micho3}, \cite{sald1}, \cite{sald2} before this paper.

\section{The quantum Hopf Fibration}

Let us take the compact matrix quantum group (cmqg) $\SU_q(2)$ (the quantum $SU(2)$ group \cite{woro2}) with $q$ $\in$ $(-1,1)-\{0\}$.  The dense $\ast$--algebra $SU_q(2)$ is generated by two letters $\{\alpha,\gamma\}$ and they satisfy
\begin{equation*}
\begin{aligned}
\alpha^{\ast}\alpha+\gamma^{\ast}\gamma=\mathbbm{1},\qquad \alpha\alpha^{\ast}+q^{2}\gamma\gamma^{\ast}=\mathbbm{1},\\
\gamma\gamma^{\ast}=\gamma^{\ast}\gamma,\quad q\gamma\alpha=\alpha\gamma, \quad q\gamma^{\ast}\alpha=\alpha\gamma^{\ast},
\end{aligned}
\end{equation*}
and its $\ast$--Hopf algebra structure is given by   
\begin{equation*}
\phi(\alpha)=\alpha\otimes\alpha-q\gamma^\ast\otimes\gamma, \qquad \phi(\gamma)=\gamma\otimes\alpha+\alpha^\ast\otimes\gamma,
\end{equation*}
\begin{equation*}
\epsilon(\alpha)=1, \qquad \epsilon(\gamma)=0,
\end{equation*}
\begin{equation*}
\kappa(\alpha)=\alpha^\ast, \quad \kappa(\alpha^\ast)=\alpha,\quad \kappa(\gamma)=-q\gamma,\quad \kappa(\gamma^\ast)=-q^{-1}\gamma^\ast.
\end{equation*}

In an abuse of notation, we will identify the group $U(1)$ with the Laurent polynomial algebra, i.e., $$U(1):=\C[z,z^\ast]=\C[z,z^{-1}]$$  and in this way, its $\ast$--Hopf algebra structure is defined by
\begin{equation*}
\phi'(z)=z\otimes z, \qquad \epsilon'(z)=1, \qquad \kappa'(z)=z^\ast, \qquad \kappa'(z^\ast)=z.
\end{equation*}
This algebra is commutative and $\kappa'$ is a $\ast$--algebra morphism.\\

By defining the $\ast$--Hopf algebra epimorphism
\begin{equation*}
j:SU_q(2) \longrightarrow U(1)
\end{equation*}
such that $j(\alpha)=z,$ $j(\gamma)=0$, we can consider
\begin{equation*}
{_{SU_q(2)}}\Phi:=(\id_{SU_q(2)}\otimes j)\circ \phi:SU_q(2)\longrightarrow SU_q(2)\otimes U(1).
\end{equation*}
Now we define the quantum $2$--sphere as (the quantum space whose $\ast$--algebra of $\C$--valued functions is given by) the $\ast$--subalgebra of $SU_q(2)$
\begin{equation*}
(\S^2_q,\cdot,\mathbbm{1},\ast),
\end{equation*}
where $$\S^2_q:=\{ x\in SU_q(2) \mid {_{SU_q(2)}}\Phi(x)=x\otimes \mathbbm{1} \}$$ which can be viewed as the $\ast$--algebra generated by $\{\alpha\alpha^\ast,\alpha\gamma^\ast\}$ and it is a quantum sphere in the sense of \cite{pod}. In this way, the quantum principal $U(1)$--bundle over $\S^2_q$ given by 
\begin{equation}
\label{f.1}
\zeta_{HF}=(SU_q(2),\S^2_q, {_{SU_q(2)}}\Phi)
\end{equation}
is usually called the {\it quantum Hopf fibration} \cite{bm}, \cite{micho2}.

Now to accomplish our purpose we will take the differential calculus on $\zeta_{HF}$ shown \cite{micho3}; however, we are going to use the notation presented in \cite{sald1}, \cite{sald2}. The bicovariant $\ast$--FODC of $U(1)$ will be denoted by $(\Gamma,d)$ and we shall take 
\begin{equation}
\label{f.2}
\beta':=\{\varsigma:=\pi'(z-z^\ast)\}
\end{equation}
as a Hamel basis of ${_\inv}\Gamma$. In this case, the quantum germs map and the adjoint corepresentation will be denoted by $\pi':U(1)\longrightarrow {_\inv}\Gamma$ and $ \ad': {_\inv}\Gamma \longrightarrow {_\inv}\Gamma\otimes U(1)$, respectively. For this space, the universal differential envelope $\ast$--calculus of $(\Gamma,d)$ (\cite{micho1}), $(\Gamma^\wedge,d,\ast)$, satisfies $$\Gamma^{\wedge k}=\{0\}\;\;\mbox{ for }\;\;k\geq 2;$$ however, it differs from the {\it classical} differential calculus on $U(1)$ \cite{micho3}.

\begin{Definition}
\label{2.1}
With the previous differential calculus, the linear map
\begin{equation*}
\begin{aligned}
\omega^\c:{_\inv}\Gamma &\longrightarrow \Omega^1(SU_q(2))\\
\theta &\longmapsto \mathbbm{1}\otimes \theta
\end{aligned}
\end{equation*}
is a real, regular and multiplicative quantum principal connection (qpc) and it is called the canonical quantum principal connection in $\zeta_{HF}$.
\end{Definition}

\begin{Proposition}
\label{2.2}
The only regular qpc is $\omega^\c$.
\end{Proposition}

\begin{proof}
According to the general theory of qpcs we know that every regular qpc has to have the form \cite{stheve}
\begin{equation*}
\omega^\c+\lambda
\end{equation*}
such that $\varphi\,\lambda(\theta)= (-1)^k \,\lambda(\theta\circ {\kappa'}^{-1}(\varphi^{(1)}))\varphi^{(0)}$ for all $\varphi$ $\in$ $\Hor^{k} SU_q(2)$, $\theta$ $\in$ ${_\inv}\Gamma$ with ${_\H}\Phi(\varphi)=\varphi^{(0)}\otimes \varphi^{(1)}$ (in Sweedler's notation). We are going to prove that $\lambda=0$. A direct calculation shows that $${_\H}\Phi(\lambda(\varsigma))=({_\H}\Phi\otimes \id_{U(1)})ad'(\varsigma)\;\Longleftrightarrow\; \lambda(\varsigma)\,\in\,\Omega^1(\S^{2}_q);$$ so $\lambda(\varsigma)=x\eta_-+y\eta_+$ with $x=\displaystyle \sum_{m+k-l=2}\lambda_{mkl}\,\alpha^m\gamma^k\gamma^{\ast l}$, $y=\displaystyle\sum_{p+q-r=-2}\,\mu_{pqr}\alpha^p\gamma^q\gamma^{\ast r},$ $\lambda_{mkl}$, $\mu_{pqr}$ $\in$ $\C$ (these elements form a Hamel basis). Due to the fact that $\lambda$ has to satisfies $ \lambda(\varsigma)\,\alpha=\alpha \,\lambda(\varsigma\circ z)=q^{-2}\alpha\,\lambda(\varsigma),$ we get  $$x\eta_-\alpha=q^{-2}\alpha x \eta_-\;\Longrightarrow\; x=\sum_{m+k-l=2}\lambda_{mkl}\,\alpha^m\gamma^k\gamma^{\ast l}\,\mbox{ with }\, k+l=1,\; m\geq 0.$$ Applying the same process to $\gamma$ we find that $m=-1$ which is a contradiction, so $x=0$. A similar calculation shows $y=0$ and hence $\lambda=0$.
\end{proof}

A quick calculation  shows  that the operator $D$ presented in \cite{micho3} is the covariant derivative of $\omega^\c$  and its curvature fulfills 
\begin{equation}
\label{f.3}
R^{\omega^\c}(\varsigma)=(1+q^2)q\,\eta_-\eta_+.
\end{equation}

It is worth mentioning that for the form of the differential calculus, the only possible embedded differential (\cite{micho2}, \cite{stheve}) is $\delta=0$ and for $q=1$, $\omega^\c$ is the principal connection associated to the Levi--Civitta connection.\\

\section{Associated Quantum Vector Bundles and the Quantum Hodge Operator}

In accordance with the general theory of associated qvbs, we need to check that Equations 34 and 35 of \cite{sald1} hold.

It is well--know that a complete set of mutually inequivalent irreducible unitary finite dimensional (smooth) $U(1)$--corepresentations $\T$ is in biyection with $\Z$. These corepresentations are unitary with the canonical inner product of $\C$ and it is worth mentioning that for all $n$ $\in$ $\Z$, the matrix of the canonical corepresentation morphism between the corepresentation and its double contragradiant is the identity matrix (see Equation 35 of \cite{sald1}).

\begin{Proposition}
\label{3.1}
For every $n$ $\in$ $\Z$, Equations 34 and 35 of \cite{sald1} are satisfied.
\end{Proposition}

\begin{proof}
By taking 
\begin{equation*}
\begin{aligned}
T^\triv :\C &\longrightarrow \S^2_q\\
w &\longmapsto w \mathbbm{1}
\end{aligned}
\end{equation*}
it follows that the statement is true for $n=0$.  Now let us take $n$ $\in$ $\N$ and consider the linear maps
\begin{equation*}
T^n_{k+1}:\C\longrightarrow \mathrm{SU}_q(2)
\end{equation*}
defined by
 $$T^n_{k+1}(1)=\left[\begin{array}{lcr}
  n  \\
k \\
\end{array}\right]^{\frac{1}{2}}_{q^{-2}}\alpha^{n-k}\gamma^k=:x^{n}_{k+1\,1} $$

\noindent with $k=0,...,n$, where $\displaystyle \left[\begin{array}{lcr}
  n  \\
k \\
\end{array}\right]_{q^{-2}}$ is the Gaussian binomial coefficient also known as the $q$--binomial coefficient \cite{m}. Due to the fact that ${_{\mathrm{SU}_q(2)}}\Phi(\alpha)=\alpha\otimes z,$ ${_{\mathrm{SU}_q(2)}}\Phi(\gamma)=\gamma\otimes z$ we get  $$T^{n}_k\;\in\;\Mor(n,{_{\mathrm{SU}_q(2)}}\Phi).$$ According to \cite{m}, these elements form the first column of the $\SU_q$--representation matrix for {\it spin} $l=\dfrac{n}{2}$, $u^l$. Since $u^{l\,\dagger}u^l=\Id_{n+1}$ $\in$ $M_{n+1}(\mathrm{SU}_q(2))$ (here $\dagger$ is denoting the transpose conjugate matrix) we get that Equation 34 holds. Taking $$Z^{n}=(q^{2(i-1)}\delta_{ij}) \;\in\; M_{n+1}(\C)$$ where $\delta_{ij}$ is the Kronecker delta,   Equation 35 holds since in this case $$W^{n\,\mathrm{T}}X^{n\,\ast}=\Id_{n_{\alpha}} \quad \mbox{ with }\quad  W^{n}=(w^{n}_{ij})=Z^{n}X^{n},\;X^{n}=(x^{n}_{k+1\,1}) $$ is the the $(1,1)$--entry of $u^lu^{l\dagger}=\Id_{n+1}$ \cite{m}. For negative integers $n$ it is enough to take the last column of $u^l$ with $l=\dfrac{|n|}{2}$ to ensure that the Equation 34 is holds and taking $$Z^{n}=(q^{-2(|n|+1-i)}\delta_{ij}) \;\in\; M_{|n|+1}(\C)$$ we get that Equation 35 holds since in this case $W^{n\,\mathrm{T}}X^{n\,\ast}=\Id_{n_{\alpha}}$ will be the $(|n|+1,|n|+1)$--entry of $u^lu^{l\dagger}=\Id_{|n|+1}$ \cite{m}.
\end{proof}

We have to remark that Equations 36 of \cite{sald1} can be viewed in terms of Hopf--Galois extension's theory \cite{qvbH}. However, an advantage of having proven Proposition \ref{3.1} is that it provides us with the left and right generators of the associated qvbs \cite{sald1}, \cite{sald2}. In this way, it is possible to take left and right associated quantum vector bundles (qvbs) and induced quantum linear connections (qlc) for all $n$ $\in$ $\Z$. The left/right qvb associated to the corepresentation $n$ will be denoted by $\zeta^\l_n:=(\Gamma^\l(\S^{2}_q,\C_n\S^{2}_q),+,\cdot)$, $\zeta^\r_n:=(\Gamma^\r(\S^{2}_q,\C_n\S^{2}_q),+,\cdot)$, respectively. Moreover, the induced qlc by any $\omega$ will be denote by $\nabla^{\omega}_n$, $\widehat{\nabla}^{\omega}_n$, respectively.

Now we have to verify if $\zeta_{HF}$ satisfies all the conditions written in Definition 2.1 and Remark 2.3 of \cite{sald2}. In order to do this, we need the following lemma.

\begin{Lemma}
\label{3.2}
Let us consider the linear functional 
\begin{equation}
\label{6.f3.4}
\begin{aligned}
\int_{\S^2_q} : \Omega^2(\S^2_q) &\longrightarrow \C\\
p\,\eta_-\eta_+ &\longmapsto h_q(p),
\end{aligned}
\end{equation}
where $h_q$ is the quantum Haar measure of $\SU_q$ \cite{woro1}. Then $$d(\Omega^1(\S^2_q)) \subseteq \Ker\left( \int_{\S^2_q} \right).$$
\end{Lemma}

\begin{proof}
Let us start by remembering the definition of $D$ for degree zero: $$D(a)=a^{(0)}(\pi_-(a^{(1)})+\pi_+(a^{(1)})),$$ where $\pi_\pm:=\rho_\pm\circ \pi$ with $\pi:\mathrm{SU}_q(2)\longrightarrow {_\inv}\Xi$ the quantum germs map and $\rho_\pm:{_\inv}\Xi\longrightarrow \C\eta_\pm$ the canonical projection. In this way, we define the linear functional $$\lambda_-: \mathrm{SU}_q(2)\longrightarrow \C $$ such that $\pi_-(a)=\lambda_-(a)\eta_-$. Notice that $\mathbbm{1}$ $\in$ $\Ker(\lambda_-)$.\\

Consider $y\eta_+$ $\in$ $\Omega^1(\S^2_q)$. Hence
\noindent \begin{eqnarray*}
\int_{\S^2_q}d(y\eta_+) \,= \,  h_q(y^{(1)})\,\lambda_-(y^{(2)})\,= \,\lambda_-(h_q(y^{(1)})y^{(2)})&= &\lambda_-(h_q\ast y)  
  \\
  &= &
\lambda_-(h_q(y)\mathbbm{1})\,= \,\lambda_-(\mathbbm{1})h_q(y)=0.
\end{eqnarray*}

\noindent In an analogous way it can be proved that $$\int_{\S^2_q}d(x\eta_-)=0$$ and therefore the Lemma follows.
\end{proof}

\begin{Proposition}
\label{3.4}
The quantum $2$--sphere satisfies all the conditions written in Definition 2.1 and Remark 2.3 of \cite{sald2} with respect to the graded differential $\ast$--algebra of base forms.
\end{Proposition}

\begin{proof}
First of all let us observe that $\S^2_q$ is (obviously)  $C^\ast$--closeable.
\begin{enumerate}
\item $\S^2_q$ is oriented since for $k> 2$, $\Omega^k(\S^2_q)=0$ and 
\begin{equation*}
\dvol:=\eta_-\eta_+
\end{equation*}
is a quantum $2$--volume form.
\item A direct calculation shows that a lqrm can be defined on $\S^2_q$ by means of
\begin{equation*}
\begin{aligned}
\langle-,-\rangle: \S^2_q\times \S^2_q &\longrightarrow \S^2_q\\
(\;\hat{p}\;,\;p\;)&\longmapsto \hat{p}\,p^\ast,
\end{aligned}
\end{equation*}
\begin{equation*}
\begin{aligned}
\langle-,-\rangle\;\;:\;\; \Omega^1(\S^2_q)\;\quad\times\;\quad \Omega^1(\S^2_q)\;\;&\longrightarrow \;\S^2_q\\
((\hat{x}\eta_-+\hat{y}\eta_+),(x\eta_-+y\eta_+))&\longmapsto \dfrac{1}{2} \left( q^2\hat{x} x^\ast + \hat{y} y^\ast\right)
\end{aligned}
\end{equation*}
and finally
\begin{equation*}
\begin{aligned}
\langle-,-\rangle: \Omega^2(\S^2_q)\;\times\; \Omega^2(\S^2_q)&\longrightarrow \S^2_q\\
(\;\hat{p}\,\dvol\;,\;p\,\dvol\;)&\longmapsto \hat{p}\,p^\ast.
\end{aligned}
\end{equation*}
With this lqrm, $\dvol$ is actually a lqr $2$--form. Taking into account Remark \ref{a.2.2}, we get a rqrm with a rqr $2$--form.
\item According to \cite{woro1}, $h_q$ is a faithful state on $\mathrm{SU}_q(2)$ and hence the linear functional of Equation \ref{6.f3.4} is actually a quantum integral. In this way, by the previous lemma we conclude that $(\S^2_q,\cdot,\mathbbm{1},\ast)$ is a quantum space without boundary.
\item A direct calculation shows
\begin{equation*}
\star_\l\, p=p^\ast\,\dvol
\end{equation*}
for all $p$ $\in$ $\S^2_q$;  
\begin{equation*}
\star_\l (p\,\dvol)=p^\ast
\end{equation*}
for all $p\,\dvol$ $\in$ $\Omega^2(\S^2_q)$ and finally  
\begin{equation*}
\star_\l\,\mu=\dfrac{1}{2}\left(-y^\ast\eta_-+x^\ast\eta_+\right),
\end{equation*}
for all $\mu=x\eta_-+y\eta_+$ $\in$ $\Omega^1(\S^2_q).$ To define $\star_\r$ we can use Remark 2.2 of \cite{sald2}.
\end{enumerate}
\end{proof}

It is worth mentioning that for this differential calculus, the only possible embedded differential is $\delta=0$ \cite{micho2}, \cite{sald2}.

\section{Yang--Mills--Scalar--Matter Fields}

In this section we are going to show solutions of the field equations of the Yang--Mills--Scalar--Matter theory (\cite{sald2}) using $(\zeta_{HF},\omega^\c)$ and all the structures that we have just defined.

\subsection{Non--commutative geometrical Yang--Mills Fields}

We know that every single qpc $\omega$ has the form \cite{micho2} $$\omega=\omega^\c+\lambda \qquad \mbox{ with }\qquad \lambda(\varsigma)=x\eta_-+y\eta_+ \,\in\,\Omega^1(\S^2_q).$$

\begin{Proposition}
\label{6.3.2}
Every YM qpc is of the form $\omega^\c+\lambda$, where $\lambda(\varsigma)=dp$ for some $p$ $\in$ $\S^{2}_q$.
\end{Proposition}

\begin{proof}
First, notice that for all qpc $\omega=\omega^\c+\lambda$ (see Equation \ref{f.3}) 
\begin{equation}
\label{6.f3.2}
R^{\omega}(\varsigma)=(1+q^2)q\eta_-\eta_++d\,\lambda(\varsigma);
\end{equation}
so 
\begin{eqnarray*}
\left.\dfrac{\partial}{\partial z}\right|_{z=0}\qS_\YM(\omega + z\,\lambda')\,&= & - \dfrac{1}{4}\, \left( \langle \lambda'(\varsigma)\,|\, d^{\star_\l}R^{\omega}(\varsigma)\rangle_\l +\langle \lambda'(\varsigma)^\ast\,|\, d^{\star_\r}R^{\omega}(\varsigma)^\ast\rangle_\r \right)
\\
&= & 
- \dfrac{1}{4}\, \left(\langle d\lambda'(\varsigma)\,|\, R^{\omega}(\varsigma)\rangle_\l+\langle d\lambda'(\varsigma)^\ast\,|\, R^{\omega}(\varsigma)^\ast\rangle_\r \right)
 \\
 &= & 
- \dfrac{1}{4}\, \left( \langle d\lambda'(\varsigma)\,|\, d\lambda(\varsigma)\rangle_\l + \langle d\lambda'(\varsigma)^\ast\,|\, d\lambda(\varsigma)^\ast\rangle_\r \right)
 \\
 &= & 
- \dfrac{1}{2}\,  \langle d\lambda'(\varsigma)\,|\, d\lambda(\varsigma)\rangle_\l.
\end{eqnarray*}
Since  $\langle -\,|\, -\rangle_\l$ is an inner product we conclude that  every YM qpc has the form $\omega^\c+\lambda$ with $d\lambda(\varsigma)=0$.

In accordance with \cite{woro2}, the zero cohomology group of $\mathrm{SU}_q(2)$ is $\C$; while the first cohomology group is $\{0\}$. Hence, since $\lambda(\varsigma)$ $\in$ $\Omega^1(\S^2_q)$ is exact, there exists $p$ $\in$ $\S^{2}_q$ such that $\lambda(\varsigma)=dp$.
\end{proof}

In this case, the quantum gauge group (qgg) of the Lagrangian (\cite{sald1}) satisfies
\begin{equation}
\label{6.f3.3}
\qGG_\YM:=\{\f\in \qGG\mid \f^{\circledast} \omega^\c=\omega^\c+\lambda \mbox{ with }d\lambda=0 \}.
\end{equation}
We can deduce that $\mathrm{U}(1) \subseteq \qGG_\YM$ and all YM qpcs are in the same orbit, just like in the {\it classical} case.

\subsection{Non--commutative geometrical n--multiple of Space--Time Scalar Matter Fields.}

 According to \cite{sald2}, it is enough to  look for eigenvectors of the left
quantum Laplace--de Rham  operator $\vartriangle_\l$ with $\S^{2}_q$--valued eigenvalues. A direct calculation shows that  $$d^{\star_\l}d p=\dfrac{1}{2}(1+q^2)^2\,p \quad\mbox{ with }\quad p=\mathbbm{1}-(1+q^2)\gamma\gamma^\ast,\,\alpha\gamma^\ast,\,\alpha^\ast\gamma.$$ It is important to mention that for $q$ $\in$ $(-1,1)-\{0\}$, these eigenvalues are not $0$. In this way, taking a potential such that $$V'=\dfrac{1}{2}(1+q^{2})^{2}$$ it is easy to find non--commutative geometrical space--time scalar matter fields. Of course, there are more solutions but they depend on the form of the potential $V$.

\subsection{Non--commutative geometrical Yang--Mills--Scalar--Matter Equations.}

Let us take $n$ $\in$ $\Z$. If $n=0$,  YMSM fields are triplets $(\omega,T_1,T_2)$ where $\omega$ is a YM qpc and $(T_1, T_2)$ is an stationary point of $\qS_\SM$.

Consider now $n\not=0$. It is easy to see that  $$d^{\star_\l}R^{\omega^\c}(\varsigma)=0;$$ so we have to look for  $T_1$ $\in$ $\Gamma^\l(\S^{2}_q,\C_n\S^{2}_q)$, $T_2$ $\in$ $\Gamma^\r(\S^{2}_q,\C_{-n}\S^{2}_q)$ such that
\begin{equation}
\label{6.f3.7}
 \langle \Upsilon_n\circ K^{\lambda}(T_1)\,|\,\nabla^{\omega^\c}_{n}T_1\rangle_\l  -  \langle \widetilde{\Upsilon}_{-n}\circ \widehat{K}^{\lambda}(T_2)\,|\,\widehat{\nabla}^{\omega^\c}_{-n}T_2\rangle_\r =0
\end{equation}
for all $\lambda$ $\in$ $\overrightarrow{\mathfrak{qpc}(\zeta_{HF})}$, and 
\begin{equation}
\label{6.f3.8}
\nabla^{\omega^\c\,\star_\l}_n\left(\nabla^{\omega^\c}_n\,T_1\right)-V'_\l(T_1)^\ast\,T_1=0\,,\qquad \widehat{\nabla}^{\omega^\c\,\star_\r}_{-n} \left(\widehat{\nabla}^{\omega^\c}_{-n}\,T_2\right)-T_2 \,V'_\r(T_2)^\ast=0.
\end{equation}

Now it is possible to explicitly find solutions. For example, for $n>0$ the triplet $(\omega^\c,T_1,T_2)$ such that $$T_1(1)=\alpha^n\,,\quad T_2(1)=\alpha^{\ast\,n}\quad \mbox{ or }\quad T_1(1)=\gamma^n\,,\quad T_2(1)=\gamma^{\ast\,n}$$ is a YMSM field for a potential such that $$V'=\dfrac{1}{2}\left(\dfrac{q^{4}(1-q^{2n})}{1-q^{2}}\right).$$ It is worth mentioning that $q\longrightarrow 1$ implies $V'\longrightarrow n$, so we recover the winding number $n$. Of course, there are more solutions; however, they depend on the form of the potential $V$.

The  spectrums of $$\nabla^{\omega^\c\,\star_\l}_n\nabla^{\omega^\c}_n:\Gamma^\l(\S^2_q,\C_n\S^2_q)\longrightarrow \Gamma^\l(\S^2_q,\C_n\S^2_q)$$ and $$\widehat{\nabla}^{\omega^\c\,\star_\r}_{n} \widehat{\nabla}^{\omega^\c}_{n}:\Gamma^\r(\S^2_q,\C_n\S^2_q)\longrightarrow \Gamma^\r(\S^2_q,\C_n\S^2_q)  $$ for all $n$ $\in$ $\Z$ are shown in the following tables. In the second row of the table $1$ and the fifth row of the table $2$, $m$, $k$ $\in$ $\N_0$ (in the other cases, $m$, $k$, $l$ $\in$ $\N$) and they  cannot be both $0$ at the same time. On the other hand, $p(\gamma^k\gamma^{\ast\,l})$, $\widehat{p}(\gamma^k\gamma^{\ast\,l})$ are polynomials with coefficients in $\C$ such that their terms are $\gamma^k\gamma^{\ast\,l}$, $\gamma^{k-1}\gamma^{\ast\,l-1}$, etc. until $\gamma$ or  $\gamma^\ast$ disappear. For example $$p(\gamma\gamma^{\ast})=\widehat{p}(\gamma\gamma^{\ast})=\mathbbm{1}-(1+q^2)\gamma\gamma^\ast.$$  Polynomials $p(\alpha^{m}\gamma^k\gamma^{\ast\,l})$, $p(\alpha^{\ast\,m}\gamma^k\gamma^{\ast\,l})$, $\widehat{p}(\alpha^{m}\gamma^k\gamma^{\ast\,l})$, $\widehat{p}(\alpha^{\ast\,m}\gamma^k\gamma^{\ast\,l})$ follow an analogous rule. For example $$p(\alpha\gamma\gamma^\ast)=-\dfrac{(q^{6}+3q^{4}+2q^{2}+1)(q^{2}+q^{4})}{q^{6}+2q^{4}+2q^{2}+1}\alpha+(q^{6}+3q^{4}+2q^{2}+1)\alpha\gamma\gamma^\ast$$ and $$\widehat{p}(\alpha\gamma\gamma^\ast)=-\dfrac{(q^{4}+2q^{2}+q^{-2}+3)(1+q^{2})}{q^{4}+2q^{2}+q^{-2}+2}\alpha+(q^{4}+2q^{2}+q^{-2}+3)\alpha\gamma\gamma^\ast. $$ In addition, let us define the {\it the $q^2$--number } $$[r]:=[r]_{q^{2}}=\dfrac{1-q^{2r}}{1-q^{2}}$$ for all $r$ $\in$ $\N$. Then let us take  $$\lambda_{m,k,l}:=\dfrac{1}{2}\left([m]\,[l+1]\,q^{2(2-l)}+[k]\,[l+1]\,q^{4+2m-2l}+[l]\,[m+1]\,q^{2(1-l)}+[l]\,[k]\,q^{4+2m-2l}\right),$$  $$\lambda_{-m,k,l}:=\dfrac{1}{2}\left([m]\,[k+1]\,q^{2(1-m)}+[l]\,[k+1]\,q^{2-2m-2l}+[k]\,[m+1]\,q^{2(2-m)}+[l]\,[k]\,q^{4-2m-2l}\right),$$  $$\widehat{\lambda}_{m,k,l}:=\dfrac{1}{2}\left([m]\,[l+1]\,q^{2-2m-2k}+[k]\,[l+1]\,q^{2(1-k)}+[l]\,[m+1]\,q^{4-2m-2k}+[l]\,[k]\,q^{2(3-k)}\right) $$ and $$\widehat{\lambda}_{-m,k,l}:=\dfrac{1}{2}\left([m]\,[k+1]\,q^{4-2k+2l}+[l]\,[k+1]\,q^{2(2-k)}+[k]\,[m+1]\,q^{2-2k+2l}+[k]\,[l]\,q^{2(1-k)}\right).$$
 
Values of the first columns form linear basis of $\mathrm{SU}_q(2)$, thus for each $n$ $\in$ $\Z$, these sections form a basis of eigenvectors.

\begin{Proposition}
Considering $\Mor(n,{_{\mathrm{SU}_q(2)}}\Phi)=\Gamma^\l(\S^2_q,\C_n\S^2_q)=\Gamma^\r(\S^2_q,\C_n\S^2_q)$ just as a vector space, the operators $\nabla^{\omega^\c\,\star_\l}_n\nabla^{\omega^\c}_n$ and $\widehat{\nabla}^{\omega^\c\,\star_\r}_n\widehat{\nabla}^{\omega^\c}_n$ are not simultaneously diagonalizable for each $n$ $\in$ $\Z$.
\end{Proposition}

\begin{proof}
We are going to prove that these operators do not commute each other. In fact $$(\widehat{\nabla}^{\omega^\c\,\star_\r}_n\widehat{\nabla}^{\omega^\c}_n)(\nabla^{\omega^\c\,\star_\l}_n\nabla^{\omega^\c}_n)T\not=(\nabla^{\omega^\c\,\star_\l}_n\nabla^{\omega^\c}_n)(\widehat{\nabla}^{\omega^\c\,\star_\r}_n\widehat{\nabla}^{\omega^\c}_n)T,$$ where $T(\mathbbm{1})=\alpha^{n}\gamma\gamma^{\ast}$ for $n>0$; $T(\mathbbm{1})=\alpha^{\ast\, n}\gamma\gamma^{\ast}$ for $n<0$ and $T(\mathbbm{1})=\alpha\gamma\gamma^{\ast\,2}$ for $n=0$.
\end{proof}

\noindent As we checked in the Subection $4.2.3$, the operators $\nabla^{\omega^\c\,\star_\l}_n\nabla^{\omega^\c}_n$ and $\widehat{\nabla}^{\omega^\c\,\star_\r}_n\widehat{\nabla}^{\omega^\c}_n$ are symmetric and non--negative.

There is a kind of $\ast$--symmetry between both operators, at least for the first five eigenvalues presented. Moreover, the eigenvalues are not symmetric under the change $n \longleftrightarrow -n,$ which is a difference with the {\it classical} case \cite{ku}. In contrast and in agreement with the {\it classical} case, both operators are not bounded. For example, let us fix $n$ and consider the eigenvalue of the fifth row of the table $6.1$ $$\dfrac{1}{2}\left([l]\,[m+1]\,q^{2(1-l)}+[m]\,[l+1]\,q^{2(2-l)}\right)=-\dfrac{q^2+q^6+2q^{2n+4}}{2(1-q^2)^2}+\dfrac{q^2(1+q^2)}{2(1-q^2)^2}(q^{-2l}+q^{2m+2}).$$ The first term in the right--hand side of the previous equality is a fixed number, and also the term $\dfrac{q^2(1+q^2)}{2(1-q^2)^2}$. However, since $m-l=n$ $$q^{-2l}+q^{2m+2}=q^{2n-2m}+q^{2m+2} \; \Longrightarrow\; \lim_{m\rightarrow \infty}q^{2n-2m}+q^{2m+2}=\pm \infty,$$ depending of the sign of $q$. By taking the {\it classical} limit $ q\longrightarrow 1$, both operators reproduces the spectrum of the Laplacian on associated vector bundles of the Hopf fibration \cite{ku}.

\begin{table}
\begin{tabular}{|c|c|c|c|c|c|}
\hline 
\multicolumn{1}{|c|}{$T(1)$} & \multicolumn{1}{|c|}{$n \in \Z$} & \multicolumn{1}{|c|}{$\lambda$}\rule[-0.3cm]{0cm}{0.8cm}\\\hline
\multicolumn{1}{|c|}{$\mathbbm{1}$} & \multicolumn{1}{|c|}{$0$} & \multicolumn{1}{|c|}{$0$} \rule[-0.3cm]{0cm}{0.8cm}\\\hline
\multicolumn{1}{|c|}{$\alpha^{m}\gamma^{k}$} & \multicolumn{1}{|c|}{$m+k=n$} & \multicolumn{1}{|c|}{$\dfrac{[n]\,q^4}{2}$}\rule[-0.4cm]{0cm}{1.2cm}\\\hline
\multicolumn{1}{|c|}{$\alpha^{\ast\,n}$, $ \gamma^{\ast\,n}$} & \multicolumn{1}{|c|}{$n>0$} & \multicolumn{1}{|c|}{$\dfrac{[n]\,q^{2(1-n)}}{2}$} \rule[-0.4cm]{0cm}{1.2cm}\\\hline
\multicolumn{1}{|c|}{$\alpha^{\ast\,m}\gamma^{\ast\,l}$} & \multicolumn{1}{|c|}{$m+l=n$} & \multicolumn{1}{|c|}{$-\dfrac{[-n]\,q^{2}}{2}$} \rule[-0.4cm]{0cm}{1.1cm}\\\hline
\multicolumn{1}{|c|}{$\alpha^m\gamma^{\ast\,l}$} & \multicolumn{1}{|c|}{$m-l=n$} & \multicolumn{1}{|c|} {$\dfrac{1}{2}\left([l]\,[m+1]\,q^{2(1-l)}+[m]\,[l+1]\,q^{2(2-l)}\right)$} \rule[-0.4cm]{0cm}{1.1cm}\\\hline
\multicolumn{1}{|c|}{$\alpha^{\ast\,m}\gamma^{k}$} & \multicolumn{1}{|c|}{$-m+k=n$} & \multicolumn{1}{|c|} {$\dfrac{1}{2} \left([m]\,[k+1]\,q^{2(1-m)}+[k]\,[m+1]\,q^{2(2-m)}\right)$} \rule[-0.4cm]{0cm}{1.1cm}\\\hline
\multicolumn{1}{|c|}{$p(\gamma^k\gamma^{\ast\,l})$} & \multicolumn{1}{|c|}{$k-l=n$} & \multicolumn{1}{|c|} {$\dfrac{1}{2} \left([l]\,q^{2(1-l)}+[k]\,q^4+2\,[l]\,[k]\,q^{2(2-l)}\right)$} \rule[-0.4cm]{0cm}{1.1cm}\\\hline
\multicolumn{1}{|c|}{$p(\alpha^{m}\gamma^k\gamma^{\ast\,l})$} & \multicolumn{1}{|c|}{$m+k-l=n$} & \multicolumn{1}{|c|} {$\lambda_{m,k,l}$} \rule[-0.4cm]{0cm}{1.1cm}\\\hline
\multicolumn{1}{|c|}{$p(\alpha^{\ast\,m}\gamma^k\gamma^{\ast\,l})$} & \multicolumn{1}{|c|}{$-m+k-l=n$} & \multicolumn{1}{|c|} {$\lambda_{-m,k,l}$} \rule[-0.4cm]{0cm}{1.1cm}\\\hline
\end{tabular}
\caption{Values for $\nabla^{\omega^\c\,\star_\l}_n\nabla^{\omega^\c}_n\, T=\lambda\, T$.}
\end{table}

\begin{center}
\begin{table}[b]
\centering
\begin{tabular}{|c|c|c|c|c|c|}
\hline 
\multicolumn{1}{|c|}{$\widehat{T}(1)$} & \multicolumn{1}{|c|}{$n \in \Z$} & \multicolumn{1}{|c|}{$\widehat{\lambda}$}\rule[-0.3cm]{0cm}{0.8cm}\\\hline
\multicolumn{1}{|c|}{$\mathbbm{1}$} & \multicolumn{1}{|c|}{$0$} & \multicolumn{1}{|c|}{$0$} \rule[-0.3cm]{0cm}{0.8cm}\\\hline
\multicolumn{1}{|c|}{$\alpha^{m}\gamma^{k}$} & \multicolumn{1}{|c|}{$m+k=n$} & \multicolumn{1}{|c|}{$-\dfrac{[-n]\,q^{2}}{2}$}\rule[-0.4cm]{0cm}{1.1cm}\\\hline
\multicolumn{1}{|c|}{$\alpha^{n}$, $ \gamma^{n}$} & \multicolumn{1}{|c|}{$n>0$} & \multicolumn{1}{|c|}{$\dfrac{[n]\,q^{2(1-n)}}{2}$} \rule[-0.4cm]{0cm}{1.1cm}\\\hline
\multicolumn{1}{|c|}{$\alpha^{\ast\,m}\gamma^{\ast\,l}$} & \multicolumn{1}{|c|}{$m+l=n$} & \multicolumn{1}{|c|}{$\dfrac{[n]\,q^{4}}{2}$} \rule[-0.4cm]{0cm}{1.1cm}\\\hline
\multicolumn{1}{|c|}{$\alpha^m\gamma^{\ast\,l}$} & \multicolumn{1}{|c|}{$m-l=n$} & \multicolumn{1}{|c|} {$\dfrac{1}{2}\left([m]\,[l+1]\,q^{2(1-m)}+[l]\,[m+1]\,q^{2(2-m)}\right)$ } \rule[-0.4cm]{0cm}{1.1cm}\\\hline
\multicolumn{1}{|c|}{$\alpha^{\ast\,m}\gamma^{k}$} & \multicolumn{1}{|c|}{$-m+k=n$} & \multicolumn{1}{|c|} {$\dfrac{1}{2}\left([k]\,[m+1]\,q^{2(1-k)}+[m]\,[k+1]\,q^{2(2-k)}\right)$} \rule[-0.4cm]{0cm}{1.1cm}\\\hline
\multicolumn{1}{|c|}{$\widehat{p}(\gamma^k\gamma^{\ast\,l})$} & \multicolumn{1}{|c|}{$k-l=n$} & \multicolumn{1}{|c|} {$\dfrac{1}{2}\left([l]\,q^{2(2-k)}+[k]\,q^{2(1-n)}+[l]\,[k]\,(1+q^{4})q^{2(1-k)}\right)$} \rule[-0.4cm]{0cm}{1.1cm}\\\hline
\multicolumn{1}{|c|}{$\widehat{p}(\alpha^{m}\gamma^k\gamma^{\ast\,l})$} & \multicolumn{1}{|c|}{$m+k-l=n$} & \multicolumn{1}{|c|} {$\widehat{\lambda}_{m,k,l}$} \rule[-0.3cm]{0cm}{0.8cm}\\\hline
\multicolumn{1}{|c|}{$\widehat{p}(\alpha^{\ast\,m}\gamma^k\gamma^{\ast\,l})$} & \multicolumn{1}{|c|}{$-m+k-l=n$} & \multicolumn{1}{|c|} {$\widehat{\lambda}_{-m,k,l}$} \rule[-0.3cm]{0cm}{0.8cm}\\\hline
\end{tabular}
\caption{Values for $\widehat{\nabla}^{\omega^\c\,\star_\r}_n\widehat{\nabla}^{\omega^\c}_n\, \widehat{T}=\widehat{\lambda}\, \widehat{T}$.}
\end{table}
\end{center}
\clearpage

\section{Concluding Comments}

There are a lot of interesting papers about the quantum Hopf fibration and its associated qvbs, as well as a treatment of gauge theory in this space, for example, \cite{qvbH}, \cite{l}, \cite{lz}, \cite{lrz}, \cite{z}. All of them follow the line of research of S. Majid and T. Brzezi{\'n}ski shown in \cite{bm}. Unlike all these papers, the work shown here follows the line of research of M. Durdevich in which we deal with two kind of covariant derivatives for any qpc (both agree in the {\it classical case}); this allows us to define induced qlc in left/right associated qvb as well as the Lagrangians and their respective field equations. For example, if we do not consider the right structure, Equation \ref{6.f3.7} becomes into $$\langle \Upsilon_n\circ K^{\lambda}(T_1)\,|\,\nabla^{\omega^\c}_{n}T_1\rangle_\l=0,$$ which does not have solutions for an arbitrary $n$. Furthermore, the operators $\nabla^{\omega^\c\,\star_\l}_n\nabla^{\omega^\c}_n,$ $\widehat{\nabla}^{\omega^\c\,\star_\r}_n\widehat{\nabla}^{\omega^\c}_n$ are not the same, they do not even commute between them!. It strongly suggests that it is important to take into account both of structures.

 The general theory is, at first sight, too restrictive in the sense that too many conditions are necessary; however this example and the others developed show that these conditions are (relatively) easy to satisfy. This theory can be applied perfectly to other spaces; there are a lot of illustrative and rich
examples to study.

In this work we focused in the canonical qpc, just because it is the {\it non--commutative counterpart} of the principal connection induced by the Levi--Civitta connection; nevertheless, there is no problem in considering other qpcs.

In terms of a physical interpretation, {\it this space models left space--time scalar matter fields and right space--time scalar antimatter fields coupled to a magnetic monopole}. Since the spectrums of $\nabla^{\omega^\c\,\star_\l}_n\nabla^{\omega^\c}_n$ and $\widehat{\nabla}^{\omega^\c\,\star_\r}_n\widehat{\nabla}^{\omega^\c}_n$ are discrete, the eigenvalues could be interpreted as {\it quantum numbers}.

Finally, it is worth mentioning that by considering all the fixed elements of  \cite{l} instead of ours (for example $\omega_\pm$ instead of $\eta_\pm$), the operator $\nabla^{\omega^\c\,\star_\l}_n\nabla^{\omega^\c}_n$ is exactly the gauge Laplacian operator shown in \cite{l} and all its results can be reproduced.

\end{document}